\documentclass[a4paper]{article}

\usepackage{amsmath}
\usepackage{amssymb}
\usepackage{amsthm}
\usepackage{graphicx}
\usepackage{bbm}
\usepackage{url}
\usepackage{color}

\newtheorem{theorem}{Theorem}
\newtheorem{proposition}[theorem]{Proposition}
\newtheorem{lemma}[theorem]{Lemma}

\newtheorem{corollary}[theorem]{Corollary}

\newcommand{\Oh}{\mathrm{O}}

\numberwithin{equation}{section}
\numberwithin{theorem}{section}

\title{Asymptotic expansion of Mathieu power series
and trigonometric Mathieu series}

\author{Stefan Gerhold\thanks{S.~Gerhold gratefully acknowledges financial support from the Austrian Science Fund (FWF) under grant P~30750 and from OeAD under grant MK 04/2018.}
\\
TU Wien\\
sgerhold@fam.tuwien.ac.at
	\and 
	\v{Z}ivorad Tomovski  \\
	Saints Cyril and Methodius University of Skopje \\
	}

\date{\today}

\begin{document}


\maketitle

\begin{abstract}
  We consider a generalized Mathieu series where the summands of the classical
  Mathieu series are multiplied by powers of a complex number.
  The Mellin transform of this series
  can be expressed by the polylogarithm or the Hurwitz zeta function. From this
  we derive a full asymptotic expansion, generalizing
  known expansions for alternating Mathieu series.
  Another asymptotic regime for trigonometric Mathieu series
  is also considered, to first order, by applying known
  results on the asymptotic behavior of trigonometric series.
\end{abstract}

\bigskip
{\bf Keywords:} Mathieu power series, trigonometric Mathieu series, asymptotic expansion, Mellin transform, polylogarithm, Hurwitz zeta function.
\medskip

{\bf MSC2010 classification:} 33E20, 41A60, 11M35.
\medskip

\section{Mathieu power series and the polylogarithm function}

In~\cite{ToPo11}, integral representations for the series
\begin{equation}\label{eq:def F}
  F_{\mu}(r,z) := \sum_{n=1}^\infty \frac{2n z^n}{(n^2+r^2)^{\mu+1}},
\end{equation}
where $r>0$, $\mu>0,$ and $z\in\mathbb{C}$ with $|z|<1$, have been established,
in terms of the Bessel function of the first kind.
The asymptotic
behavior of~\eqref{eq:def F} as $r\uparrow\infty$
has not been investigated so far, except for special values of~$z$.
For $z=1$, this series becomes the generalized Mathieu series studied in~\cite{GeHuTo19,Pa13},
with positive summands and growth order $r^{-2\mu}$ as $r\uparrow \infty$.
(Those papers also contain many further references on Mathieu series and their significance.)
For any other number $z$ on the complex unit circle, the oscillating character
of the summands causes cancellations
that make the sum decay faster, of order $r^{-2\mu-2}$. So far, this was only known
for $z=-1$ (alternating Mathieu series); see~\cite{PoSrTo06} for $\mu=1$
 and \cite{Pa13,Za09b} for general~$\mu.$
 For $|z|<1$, the leading term is of order $r^{-2\mu-2}$, too.
 As in~\cite{GeHuTo19,Pa13,Pa16,Za09b}, we use a Mellin transform approach
to expand~\eqref{eq:def F} for $r\uparrow \infty$.
The Mellin transform of~$F_{\mu}(r,z)$ can be expressed by the
polylogarithm function
\begin{equation}\label{eq:def li}
  \mathrm{Li}_\alpha(z)  = \sum_{n=1}^\infty \frac{z^n}{n^\alpha},
  \quad |z|<1,\ \alpha\in\mathbb C.
\end{equation}
It is well known that the polylogarithm has an analytic continuation, e.g.\ by the Lindel\"of
integral\footnote{The name $\mathrm{Li}$ does not originate from \emph{Lindel\"of integral},
but rather from \emph{logarithmic integral} (see~\cite{FlSe09}).}
\begin{equation}\label{eq:lind}
  \mathrm{Li}_\alpha(z) = -\frac{1}{2\pi i} \int_{1/2-i\infty}^{1/2+i\infty}
    \frac{(-z)^u}{u^\alpha} \frac{\pi}{\sin \pi u}du.
\end{equation}
{}From this representation and the definition~\eqref{eq:def li}, it is easy to see that $\mathrm{Li}_\alpha(z)$ is an entire
function of~$\alpha$ for any  $z\in\mathbb{C}\setminus[1,\infty).$
See~\cite{Fl99,FlGeSa10} and p.~409 in~\cite{FlSe09} for details. 
In our main result, we need an estimate for $\mathrm{Li}_\alpha(z)$
for fixed~$z$ and large $\mathrm{Im}(\alpha)$. This will be established in Section~\ref{se:est},
using the well-known representation of $\mathrm{Li}_\alpha(z)$ by
the Hurwitz zeta function. 
Moreover, we will require the following complex extension of Abel's convergence theorem
(Stolz 1875); see p.~406 in~\cite{Kn51}.
\begin{theorem}\label{thm:stolz}
  Let $\sum_{n=0}^\infty a_n z^n$ be a complex power series with radius of convergence~$1$.
  If this series converges at a point~$z_0$ of the unit circle, then
  \[
     \lim_{\substack{z\to z_0 \\ z\in\Delta}} \sum_{n=0}^\infty a_n z^n = \sum_{n=0}^\infty a_n z_0^n,
  \]
  where $\Delta$ is any triangle in the unit disk
  with $z_0$ as one of its vertices.
\end{theorem}
This theorem implies consistency of~\eqref{eq:def li} with the analytic continuation of
the polylogarithm, i.e.\
that $\sum_{n=1}^\infty n^{-\alpha}e^{i nx}=\mathrm{Li}_\alpha(e^{ix})$
for $\mathrm{Re}(\alpha)>1$ and $x\in(0,2\pi)$, which will be used below.
  (This actually holds for
$\mathrm{Re}(\alpha)>0$, see p.~401 in~\cite{Kn51} for convergence of
$\sum n^{-\alpha}e^{i nx}$, but we do not need this fact.)%

\section{Main result}

\begin{theorem}\label{thm:main}
  Let $\mu>0$ and $1\neq z\in\mathbb{C}$ with $|z|\leq1$.
  As $r\uparrow \infty$, we have the asymptotic expansion
  \begin{align}
    F_{\mu}(r,z) &\sim 
     \sum_{k=0}^\infty r^{-2k-2\mu-2} \frac{2(-1)^{k}\,\Gamma(k+\mu+1)}{k!\,\Gamma(\mu+1)}\,
         \mathrm{Li}_{-2k-1}(z)  \label{eq:F expans} \\
     &= 2\sum_{k=0}^\infty r^{-2k-2\mu-2} (-1)^{k}\binom{k+\mu}{k}
         \mathrm{Li}_{-2k-1}(z),
  \end{align}
  where $\mathrm{Li}_{-2k-1}(z)$ is defined by~\eqref{eq:def li} or~\eqref{eq:lind}.
\end{theorem}
\begin{proof}
  The Mellin transform of~$F_{\mu}(r,z)$ w.r.t.~$r$ is (cf.~\cite{GeHuTo19,Pa13})
  \begin{align}
  (\mathcal{M} F_{\mu})(u) &= \int_0^\infty r^{u-1} F_{\mu}(r,z)dr  \notag \\
     &= \sum_{n=1}^\infty 2n z^n \int_{0}^\infty
      \frac{r^{u-1}}{ (n^2+r^2)^{\mu+1}}dr \notag \\
    &= \frac{\Gamma(\mu+1-u/2)\Gamma(u/2)}{\Gamma(\mu+1)} \sum_{n=1}^\infty n^{u-2\mu-1} z^n \notag \\
    &= \frac{\Gamma(\mu+1-u/2)\Gamma(u/2)}{2\Gamma(\mu+1)} \mathrm{Li}_{2\mu+1-u}(z),
    \quad 0 < \mathrm{Re}(u) < 2\mu. \label{eq:MF}
  \end{align}
  For $|z|<1$, the last equality is clear from~\eqref{eq:def li}.
  For $z\neq1$ with $|z|=1$, we have
  \[
    \mathrm{Li}_{2\mu+1-u}(z)
    = \lim_{\substack{w\to z \\ w\in\Delta}} \sum_{n=1}^\infty n^{u-2\mu-1} w^n
    = \sum_{n=1}^\infty n^{u-2\mu-1} z^n
  \]
  for $0 < \mathrm{Re}(u) < 2\mu,$   where the first equality is clear from
  analytic continuation and the second one from Theorem~\ref{thm:stolz},
  with $\Delta$ as in that theorem.   
  
  As mentioned above,   $\mathrm{Li}_{2\mu+1-u}(z)$ is an
  entire function of~$u$. 
  Thus, $\mathcal{M} F_{\mu}$ is a meromorphic function.
  For the desired asymptotic expansion ($r\uparrow \infty$), the poles in the right half-plane are the
  relevant ones. They are those of the factor $\Gamma(\mu+1-u/2)$,
  located at $2k+2\mu+2$ for  $k\in\mathbb{N}_0$.
  We can now use the standard procedure of expanding a function whose Mellin
  transform is meromorphic (see, e.g., \cite{FlGoDu95} or Section 4.1.1 in~\cite{PaKa01}).
  To justify Mellin inversion, we have to argue that
  $(\mathcal{M} F_{\mu})(\mathrm{Re}(u)+i\cdot)$ is integrable.
  By Stirling's formula (see p.~224 in~\cite{Co35}), we have
  \begin{equation}\label{eq:stirling}
    |\Gamma(w)| \sim \sqrt{2\pi}\, |\mathrm{Im}(w)|^{\mathrm{Re}(w)-1/2}\exp(-\tfrac12 \pi |\mathrm{Im}(w)|)
  \end{equation}
  for $\mathrm{Re}(w)$ bounded and $|\mathrm{Im}(w)|\uparrow \infty$.
  This implies
  \[
    \Gamma(\mu+1-u/2)\Gamma(u/2) = \Oh\Big(\exp\big(-\tfrac12\pi |\mathrm{Im}(u)|\big)\, |\mathrm{Im}(u)|^\mu\Big)
  \]
  for $\mathrm{Re}(u)$ bounded and $|\mathrm{Im}(u)| \uparrow  \infty.$
  Using this and Proposition~\ref{prop:est} below, we see from~\eqref{eq:MF}
  that $(\mathcal{M}F_{\mu})(u)$ decays exponentially along vertical lines,
  \[
    (\mathcal{M}F_{\mu})(u) = \Oh\big(\exp(-\varepsilon |\mathrm{Im}(u)|)\big),
  \]
  and is thus integrable for $\mathrm{Re}(u)>0$,
  as long as the vertical contour avoids the poles of $\Gamma(\mu+1-u/2)$.
  The Mellin inversion formula then says that
  \[
    F_{\mu}(r,z) = \frac{1}{2\pi i}\int_{u_0-i\infty}^{u_0+i\infty}
    r^{-u} (\mathcal{M}F_{\mu})(u)du,
    \quad 0<u_0<2\mu+2.
  \]
  The above locally uniform estimate for $\mathcal{M}F_{\mu}$ allows to push the contour
  to the right, and the residue theorem yields the expansion
  \begin{align*}
    F_{\mu}(r,z) &\sim - \sum_{k=0}^\infty \mathrm{res}_{u=2k+2\mu+2}\ r^{-u}(\mathcal{M}F_{\mu})(u) \\
      & = - \sum_{k=0}^\infty r^{-2k-2\mu-2}\ \mathrm{res}_{u=2k+2\mu+2}\ (\mathcal{M}F_{\mu})(u)\\
      &= - \sum_{k=0}^\infty r^{-2k-2\mu-2} \frac{\Gamma(k+\mu+1)}{\Gamma(\mu+1)}\,
        \mathrm{Li}_{-2k-1}(z) \
        \mathrm{res}_{u=2k+2\mu+2}\ \Gamma(\mu+1-u/2).
  \end{align*}
  It easily follows from $\mathrm{res}_{s=-k}\, \Gamma(s)=(-1)^k/k!,$ $k\in\mathbb{N}_0$,
  that
  \[
    \mathrm{res}_{u=2k+2\mu+2}\ \Gamma(\mu+1-u/2) = \frac{2(-1)^{k+1}}{k!},
    \quad k\in\mathbb{N}_0.
  \]
  This implies the result. 
\end{proof}
In Section~\ref{se:trig} we will comment on the relation between Theorem~\ref{thm:main}
and some results from the literature on Mathieu series.

\section{Estimates for the polylogarithm  and the Hurwitz zeta function}\label{se:est}

The Hurwitz zeta function is defined by
\begin{equation}\label{eq:hurwitz}
    \zeta(s,q) = \sum_{n=0}^\infty \frac{1}{(q+n)^s}, \quad \mathrm{Re}(s) >1, \mathrm{Re}(q) >0,
\end{equation}
and can be extended to $s\in\mathbb{C}\setminus\{1\}$ by analytic continuation.
It is related to the polylogarithm by Jonqui{\`e}re's formula~\cite{Jo89}
\begin{equation}\label{eq:Li Hurwitz}
    \mathrm{Li}_{\alpha}(z) = \frac{\Gamma(1-\alpha)}{(2\pi)^{1-\alpha}}
      \bigg( i^{1-\alpha} \zeta\Big(1-\alpha,\frac12 + \frac{\log(-z)}{2\pi i}\Big)
      + i^{\alpha-1} \zeta\Big(1-\alpha,\frac12 - \frac{\log(-z)}{2\pi i}\Big)\bigg),
\end{equation}
 valid for $z\in\mathbb{C}\setminus[0,\infty)$
and $\alpha\in\mathbb{C}\setminus\mathbb{N}_0$. To ensure integrability of the
Mellin transform in the proof of Theorem~\ref{thm:main}, we need a growth estimate
for $\mathrm{Li}_{\alpha}(z)$, or equivalently for $\zeta(s,q)$, for large $\mathrm{Im}(\alpha)$
resp.\ $\mathrm{Im}(s).$
 A related estimate for the polylogarithm
occurs in Lemma~2 of~\cite{DrGe14}.
%
\begin{proposition}\label{prop:est}
  Let $\mathrm{Re}(q)>0$ and $\theta_1,\theta_2>0$. Then there is $\varepsilon>0$ such that
  \begin{equation}\label{eq:zeta exp}
     \zeta(s,q) = \Oh\Big( \exp\big((\tfrac12 \pi -\varepsilon)|\mathrm{Im}(s)|\big) \Big)
  \end{equation}
  as $|\mathrm{Im}(s)| \uparrow \infty$, uniformly w.r.t.\
  $-\theta_1\leq \mathrm{Re}(s)\leq \theta_2.$
  Similarly, for  $z\in\mathbb{C}\setminus[1,\infty)$ and $\theta_1,\theta_2>0$ there is $\varepsilon>0$ such that
  \begin{equation}\label{eq:li exp}
    \mathrm{Li}_{\alpha}(z) =
    \Oh\Big( \exp\big((\tfrac12 \pi -\varepsilon)|\mathrm{Im}(\alpha)|\big) \Big)
  \end{equation}
  for $-\theta_1\leq \mathrm{Re}(\alpha)\leq \theta_2.$
\end{proposition}
The proposition will be proved at the end of this section.
For $q\in(0,1]$, \eqref{eq:zeta exp}
can be strengthened to a \emph{polynomial} bound by \S 13.5 in~\cite{WhWa96},
used in the Mathieu series context in~\cite{Za09b}.
This easily yields a polynomial bound instead of~\eqref{eq:li exp} under the additional
assumption that $|z|=1$  (see the proof of
Proposition~\ref{prop:est} below). We also mention that,
for $q\in(0,1]$ and $\mathrm{Re}(s)\in[\tfrac12,1],$ rather tight polynomial
bounds for $\zeta(s,q)$ have been obtained~\cite{Fo02,Ku99,Ri67,vdKo30}.
However, for non-real~$q$,
it is not obvious how to adapt \S 13.5 in~\cite{WhWa96}, which uses
Hurwitz' Fourier series for $\zeta(s,q)$. We thus use a different approach to prove~\eqref{eq:zeta exp},
similar to p.~271 in~\cite{Pr57}.
\begin{lemma}
  Let $\mathrm{Re}(q)>0$ and $\theta_1,\theta_2>0$.  Then
  \begin{equation}\label{eq:zeta sum}
    \zeta(s,q) = \sum_{n=0}^{\lfloor | \mathrm{Im}(s) | \rfloor} (n+q)^{-s} \
    \, +\Oh\big( | \mathrm{Im}(s) |^{\theta_1+1} \big)
  \end{equation}
  as $|\mathrm{Im}(s)| \uparrow \infty$, uniformly w.r.t.\
  $-\theta_1\leq \mathrm{Re}(s)\leq \theta_2.$
\end{lemma}
\begin{proof}
  Define $k:=\lfloor \theta_1/2 \rfloor+1$ and $f(x):=(x+q)^{-s}$.
  By the Euler--Maclaurin formula, we have
  \begin{multline*}
    \sum_{n=a}^b (n+q)^{-s} = \int_a^b f(x)dx + \frac12\big(f(a)+f(b)\big)\\
      + \sum_{j=1}^k \frac{B_{2j}}{(2j)!}\big(f^{(2j-1)}(b)-  f^{(2j-1)}(a)\big)
      +\int_a^b \frac{B_{2k+1}(x-\lfloor x\rfloor)}{(2k+1)!} f^{(2k+1)}(x)dx
  \end{multline*}
  for $\mathrm{Re}(s)>1,$ where $a \leq b \in\mathbb{N}_0$. As usual, the $B$s denote
  Bernoulli numbers resp.\ polynomials.
  Since
  \[
    f^{(m)}(x) = (-1)^m (s)_m (x+q)^{-s-m}, \quad m\geq0,
  \]
  where $(s)_m =s(s+1)\cdots(s+m-1)$ is the Pochhammer symbol, we obtain
  \begin{multline}\label{eq:z est}
    \zeta(s,q)= \sum_{n=0}^{a-1} (n+q)^{-s} + \frac{(a+q)^{1-s}}{s-1}
    +\frac12(a+q)^{-s} \\
    + \sum_{j=1}^k \frac{B_{2j}}{(2j)!} (s)_{2j-1}(a+q)^{-s-2j+1} \\
    -\frac{ (s)_{2k+1}}{(2k+1)!} \int_a^\infty \frac{B_{2k+1}(x-\lfloor x\rfloor)}
      {(x+q)^{s+2k+1}} dx
  \end{multline}
  for $\mathrm{Re}(s)>1$. By analytic continuation, this equality
  extends to $s\neq1$ with $\mathrm{Re}(s)>-2k$.
  We now put
  \[
    a := \lfloor | \mathrm{Im}(s) | \rfloor
  \]
  and consider $a\uparrow\infty$ with the specified
   restriction $-\theta_1\leq\mathrm{Re}(s)\leq\theta_2$.
  Note that, by definition of $k$,
  \[
    -2k = -2\lfloor \theta_1/2 \rfloor -2 < -2(\theta_1/2-1)-2 = -\theta_1,
  \]
  and so~\eqref{eq:z est} holds in a sufficiently large half-plane.
  Since
  \[
    \arg(a+q) = \arctan \frac{\mathrm{Im}(q)}{a + \mathrm{Re}(q)} = \Oh(1/a),
  \]
  we have for any $m\in\{-2k-1,\dots,0,1\}$
  \begin{align}\label{eq:a+q est}
    |(a+q)^{-s+m}|
    =|a+q|^{m-\mathrm{Re}(s)} e^{\mathrm{Im}(s) \arg(a+q)}
    =\Oh(a^{m-\mathrm{Re}(s)}).
  \end{align}
  As $\mathrm{Re}(s)$ is bounded, we have $(s)_m=\Oh(a^m)$. We use this, \eqref{eq:a+q est}, and the boundedness of $B_{2k+1}(x-\lfloor x\rfloor)$ in~\eqref{eq:z est} to get
  \begin{align*}
    \zeta(s,q) &= \sum_{n=0}^{a-1} (n+q)^{-s} + \Oh(a^{1-\mathrm{Re}(s)}) \\
    &= \sum_{n=0}^{a-1} (n+q)^{-s} + \Oh(a^{\theta_1+1}). \qedhere
  \end{align*}
\end{proof}
We now prove Proposition~\ref{prop:est}, using a crude estimate for the sum in~\eqref{eq:zeta sum},
which suffices for our purposes.
\begin{proof}[Proof of Proposition~\ref{prop:est}]
  The sum in~\eqref{eq:zeta sum} can be estimated by
  \[
    \sum_{0\leq n\leq |\mathrm{Im}(s)|}| (n+q)^{-s}|
    = \sum_{0\leq n\leq |\mathrm{Im}(s)|} |n+q|^{-\mathrm{Re}(s)}
      e^{\mathrm{Im}(s) \arg(n+q)}.
  \]
  The factor
  \[
    |n+q|^{-\mathrm{Re}(s)} \leq (|\mathrm{Im}(s)|+|q|)^{-\mathrm{Re}(s)}
    \, \vee\, |q|^{-\mathrm{Re}(s)}
  \]
  is of at most polynomial growth. Note that any polynomial factor does not
  affect the validity of~\eqref{eq:zeta exp},
  by possibly shrinking~$\varepsilon$. Since $\mathrm{Re}(q)>0,$ we have
  \[
    |\arg(n+q)|\leq |\arg(q)|<\tfrac12 \pi,\quad n\geq0.
  \]
  This proves~\eqref{eq:zeta exp}. It remains to prove~\eqref{eq:li exp}. 
  For $z\in[0,1)$, we obviously have $|\mathrm{Li}_\alpha(z)|\leq \zeta(\alpha)$.
  The Riemann zeta function $\zeta(\cdot)=\zeta(\cdot,1)$ is of at most polynomial growth in any
  right half-plane (see p.~95 in~\cite{Ti86}), and so we may from now on assume
  $z\in\mathbb{C}\setminus[0,\infty)$ and apply~\eqref{eq:Li Hurwitz}, with
  \[
    \mathrm{Re}(q_\pm) = \mathrm{Re}\Big(\frac12 \pm \frac{\log(-z)}{2\pi i}\Big)>0.
  \]
   The factor $(2\pi)^{\alpha-1}$ in~\eqref{eq:Li Hurwitz} is clearly $\Oh(1)$, and
  \[
    |i^{\pm(1-\alpha)}| = \exp\big(\pm \tfrac12 \pi \mathrm{Im}(\alpha) \big).
  \]
  By Stirling's formula (see~\eqref{eq:stirling}), we have
  \[
     \Gamma(1-\alpha) = \Oh\Big(|\mathrm{Im}(\alpha)|^{1/2-\mathrm{Re}(\alpha)}
       \exp(-\tfrac12 \pi |\mathrm{Im}(\alpha)|\big)\Big).
  \]
  Therefore, the exponential estimates contributed by $ \Gamma(1-\alpha)$
  and $i^{1-\alpha}$ in~\eqref{eq:Li Hurwitz} cancel, and using~\eqref{eq:zeta exp}
  in~\eqref{eq:Li Hurwitz} proves~\eqref{eq:li exp}. 
\end{proof}

\section{Trigonometric Mathieu series}\label{se:trig}

When $z=e^{ix}$ lies on the unit circle, then the real resp.\ imaginary
part of~\eqref{eq:def F} become  Mathieu cosine resp.\ sine series.
These, and their partial sums, were considered in~\cite{SrToLe15}. In particular,
several inequalities for trignometric Mathieu series were proved there.
For asymptotics in the large~$r$ regime,
the following result immediately follows from Theorem~\ref{thm:main}, by putting $z=e^{ix}.$
\begin{corollary}
Let $\mu>0$ and $x\in(0,2\pi)$.
  As $r\uparrow \infty$, we have the asymptotic expansions
  \begin{equation}\label{eq:cos}
    \sum_{n=1}^\infty \frac{2n \cos(nx)}{(n^2+r^2)^{\mu+1}} \sim 
     \sum_{k=0}^\infty r^{-2k-2\mu-2} \frac{2(-1)^{k}\,\Gamma(k+\mu+1)}{k!\,\Gamma(\mu+1)}\,
         \mathrm{Re}\big(\mathrm{Li}_{-2k-1}(e^{ix})\big)
  \end{equation}
  and
  \begin{equation}\label{eq:sin}
    \sum_{n=1}^\infty \frac{2n \sin(nx)}{(n^2+r^2)^{\mu+1}} \sim 
     \sum_{k=0}^\infty r^{-2k-2\mu-2} \frac{2(-1)^{k}\,\Gamma(k+\mu+1)}{k!\,\Gamma(\mu+1)}\,
         \mathrm{Im}\big(\mathrm{Li}_{-2k-1}(e^{ix})\big).
  \end{equation}
\end{corollary}
In particular, setting $x=\pi$ in~\eqref{eq:cos} yields
 the alternating Mathieu series treated
 in~\cite{Pa13} and~\cite{Za09b}. It can be easily checked that this special case
 of~\eqref{eq:cos} is consistent
 with~(2.7) in~\cite{Pa13}. To see this, note that
 \[
   \mathrm{Li}_{-2k-1}(-1) = (2^{2k+2}-1) \zeta(-2k-1),
 \]
 which follows from the basic relation
 \[
   \eta(s) = (1-2^{1-s})\zeta(s)
 \]
 between the Riemann zeta function and the Dirichlet eta function
 $\eta(s)=\sum_{n=1}^\infty (-1)^{n+1}n^{-s}.$
 Moreover, the parameter~$\mu$ from~\cite{Pa13} is our $\mu+1$, and there
 is a typo in~(2.7) of~\cite{Pa13}: the summation should start at $k=0$.
 (Besides, the last sum at the bottom of p.~6213 in~\cite{Pa13} should be multiplied by~$-1$.)

More generally, if $x$ in~\eqref{eq:cos} is a rational multiple of~$\pi,$ we can split the
trigonometric Mathieu series
into finitely many segments to which the following result from~\cite{Za09b} can be applied.
\begin{theorem}\label{thm:zast}
  For $a>0$, $\gamma\in\mathbb R$, $\alpha>0$, $\mu>\max\{(\gamma+1)/\alpha,0\}$, and
  $-(\gamma+1)/\alpha \notin \mathbb{N}_0,$ we have
  \begin{multline*}
    \sum_{\nu=0}^\infty \frac{(\nu+a)^\gamma}{(y(\nu+a)^\alpha+1)^\mu}
    \sim
    \frac{\Gamma\big(\tfrac{\gamma+1}{\alpha}\big)\Gamma\big(\mu-\tfrac{\gamma+1}{\alpha}\big)}
      {\alpha\Gamma(\mu)} y^{-\frac{\gamma+1}{\alpha}} \\
      +\sum_{k=0}^\infty \frac{(-1)^k}{k!} \frac{\Gamma(\mu+k)}{\Gamma(\mu)}
      \zeta(-\alpha k-\gamma,a)\, y^k, \quad y\downarrow 0.
  \end{multline*}
\end{theorem}
To see that Theorem~\ref{thm:zast} gives an alternative proof of~\eqref{eq:cos}
for $x$ a rational multiple of~$\pi$,
let $x=p\pi/q$  with
$p/q\in\mathbb{Q}\cap(0,2)$. Then, putting $y:=(2q/r)^2,$
 we can write the left hand side of~\eqref{eq:cos} as
\begin{align}
  &\sum_{n=1}^\infty \frac{2n \cos(nx)}{(n^2+r^2)^{\mu+1}} = \sum_{m=0}^{2q-1} \sum_{\nu=0}^\infty
    \frac{2(2\nu q+m)}{\big((2\nu q+m)^2+r^2\big)^{\mu+1}} \cos\Big(2\nu p\pi+\frac{mp \pi}{q} \Big)\notag\\
    &= 4qr^{-2\mu-2}\bigg( \sum_{\nu=0}^\infty \frac{\nu+1}{\big(y(\nu +1)^2+1\big)^{\mu+1}} +\notag \\
     &\quad   \sum_{m=1}^{2q-1} \cos\Big(\frac{mp \pi}{q}\Big)
      \sum_{\nu=0}^\infty \frac{\nu+m/2q}{\big(y(\nu +m/2q)^2+1\big)^{\mu+1}} \bigg)\notag\\
    &\sim 4qr^{-2\mu-2}\Bigg(\frac{\Gamma(\mu)}{2\Gamma(\mu+1)}y^{-1}
    +\sum_{k=0}^\infty \frac{(-1)^k}{k!}\frac{\Gamma(k+\mu+1)}{\Gamma(\mu+1)}
    \zeta(-2k-1)y^k +\notag \\
    &\quad \sum_{m=1}^{2q-1} \cos\Big(\frac{mp \pi}{q}\Big) \Big(
       \frac{\Gamma(\mu)}{2\Gamma(\mu+1)}y^{-1} 
       +\sum_{k=0}^\infty \frac{(-1)^k}{k!}\frac{\Gamma(k+\mu+1)}{\Gamma(\mu+1)}
    \zeta(-2k-1,\tfrac{m}{2q})y^k
    \Big) \Bigg) \notag \\
    &= 2^{2k+2}q^{2k+1} \sum_{k=0}^\infty r^{-2k-2\mu-2}\frac{(-1)^k}{k!}\frac{\Gamma(k+\mu+1)}{\Gamma(\mu+1)} \Big( \zeta(-2k-1) + \notag \\
    &\qquad  \sum_{m=1}^{2q-1} \cos\Big(\frac{mp \pi}{q}\Big)
    \zeta(-2k-1,\tfrac{m}{2q}) \Big),
    \label{eq:for Zast}
\end{align}
where the asymptotic expansion follows from Theorem~\ref{thm:zast}, with $\mu$
replaced by $\mu+1$, $\gamma=1$, $\alpha=2,$ and $a=1$ resp.\ $a=2m/q,$
and
\[
  \sum_{m=0}^{2q-1}\cos\Big(\frac{mp\pi}{q}\Big)=0
\]
was used in the last equality.
Now~\eqref{eq:cos} for $x=p\pi/q$ easily follows from~\eqref{eq:for Zast}.
Note that
\[
   \mathrm{Li}_{\alpha}(e^{i p\pi/q})
  =(2q)^{-\alpha}\zeta(\alpha) + (2q)^{-\alpha} \sum_{m=1}^{2q-1} e^{imp\pi/q}
  \zeta(\alpha,m/2q), \quad \alpha\neq 1,
\]
which we apply with $\alpha=-2k-1$.
The latter (well-known) identity easily follows from~\eqref{eq:def li} and~\eqref{eq:hurwitz}
by analytic continuation. Clearly, the sine series~\eqref{eq:sin} can
be treated analogously.

We now comment on a different asymptotic regime for trigonometric Mathieu
series, namely $x\downarrow0$ for $r>0$ fixed.
For the series in~\eqref{eq:cos} and~\eqref{eq:sin}, such an expansion can be obtained using again
the Mellin transform approach. However, this would require an analysis of the singularity
structure of the Dirichlet series $\sum_{k=1}^\infty k^{-s}(k^2+r^2)^{-(\mu+1)}$,
which is doable, but deferred to future work. First order asymptotics, however, follow from
known results, even for the more general Mathieu-type sine series
\begin{equation}\label{eq:math sin}
  \tilde{S}_{\alpha,\beta,\mu}^{\gamma,\delta}(r,x):=
  \sum_{n=2}^\infty \frac{n^\alpha (\log n)^\gamma\sin(nx)}
    {\big(n^\beta (\log n)^\delta+r^2\big)^{\mu+1}}.
\end{equation}
The corresponding series without the factor $\sin(nx)$ was introduced in~\cite{GeHuTo19}.
The coefficients of the series~\eqref{eq:math sin} behave roughly like $n^{\alpha-\beta(\mu+1)},$
and the asymptotic behavior is markedly different for $\alpha-\beta(\mu+1)<-2$
and  $\alpha-\beta(\mu+1)>-2$.
\begin{proposition}\label{prop:sin}
  Let $\mu,r\geq0$ and $\alpha,\beta,\gamma,\delta\in\mathbb{R}$ such that
  \[
    0 \leq \theta:=\beta(\mu+1)-\alpha <2.
  \]
  If $\theta=0$, then we assume that  $\gamma-\delta(\mu+1)<0.$
  For $x\downarrow0,$ we have
  \[
    \tilde{S}_{\alpha,\beta,\mu}^{\gamma,\delta}(r,x)  \sim
    \begin{cases}
      \displaystyle{ \frac{\pi}{2\Gamma(\theta)\sin( \pi \theta/2)} x^{\theta-1}
      \Big(\log \frac1x\Big)^{\gamma-\delta(\mu+1)}} & 0 < \theta < 2, \\[12pt]
       \displaystyle{\frac1x \Big(\log \frac1x\Big)^{\gamma-\delta(\mu+1)}} & \theta =0.
    \end{cases}
  \]
\end{proposition}
\begin{proof}
  As the coefficient sequence
  \begin{equation}\label{eq:an}
    a_n = \frac{n^\alpha (\log n)^\gamma}
    {\big(n^\beta (\log n)^\delta+r^2\big)^{\mu+1}}
  \end{equation}
  eventually decreases, the series is convergent for $x\in(0,\pi)$ (see p.~3 in~\cite{Zy55}).
  First assume $\theta=0$.
   Since powers of logarithms are slowly varying, and asymptotic equivalence
  preserves slow variation, the sequence
  \[
    a_n \sim (\log n)^{\gamma-\delta(\mu+1)}, \quad n\to\infty,
  \]
  is slowly varying.
  Moreover, it is easy to see that the second derivative of~\eqref{eq:an}
  does not change sign for~$n$ sufficiently large, and so~$a_n$ is eventually
  convex or eventually concave.
  We can thus apply the corollary on p.~48 of Telyakovski\u{\i}~\cite{Te95} (with~$a_n$ replaced by $-a_n$
  in case of concavity) to conclude
  that
  \[
    \tilde{S}_{\alpha,\beta,\mu}^{\gamma,\delta}(r,x) \sim a_m/x, \quad x\downarrow 0,\
    x \in \Big(\frac{\pi}{m+1},\frac{\pi}{m}\Big],
  \]
  where
    we note that in~\cite{Te95} asymptotic equivalence
   is denoted by~$\approx$ and not by~$\sim$. Since $m=\lceil \pi/x \rceil=\pi/x+\mathrm{O}(1)$, the
   statement easily follows.
   
   For $\theta>0$, the sequence
   \[
     a_n \sim n^{-\theta} (\log n)^{\gamma-\delta(\mu+1)}, \quad n\to\infty,
   \]
   is \emph{regularly} varying, and our assertion is an immediate consequence
   of Theorem~1 in~\cite{AlBoTo56}. Note that it is easy to see that~$n^\theta a_n$ is eventually
   decreasing, as assumed in that theorem. However, the required implication is true even without
   this condition; see~\cite{BoSe18} for this, and for further references.
\end{proof}
If $\theta=\beta(\mu+1)-\alpha$ is larger than~$2$, we can use a
a result of Hartman and Wintner~\cite{HaWi53}. Unlike Proposition~\ref{prop:sin},
the parameter~$r$ now appears in the first asymptotic term, and there is no
$\log x$ term. Essentially, the series now converges fast enough to justify
exchanging summation and the asymptotic equivalence $\sin(nx)\sim nx$.
\begin{proposition}\label{prop:ha wi}
 Let $\mu,r\geq0$ and $\alpha,\beta,\gamma,\delta\in\mathbb{R}$ such that
    $\theta=\beta(\mu+1)-\alpha >2.$
 Then we have
 \[
   \tilde{S}_{\alpha,\beta,\mu}^{\gamma,\delta}(r,x)
   \sim
   x\, \sum_{n=2}^\infty \frac{n^{\alpha+1} (\log n)^\gamma}
    {\big(n^\beta (\log n)^\delta+r^2\big)^{\mu+1}},
   \quad x\downarrow 0.
 \]
\end{proposition}
\begin{proof}
  According to~\cite{HaWi53}, for a decreasing sequence $a_n\downarrow 0$ with $\sum na_n<\infty$, we 
  have
  \[
    \sum_{n=1}^\infty a_n \sin(nx) \sim x \sum_{n=1}^\infty n a_n, \quad
    x\downarrow 0.
  \]
  The sequence~\eqref{eq:an} decreases for large~$n$, say for $n\geq n_0$.
  By considering the sequence $\tilde{a}_n = a_{n_0}\mathbf{1}_{n\leq n_0}
  + a_{n}\mathbf{1}_{n> n_0}, $ it is very easy to see that ``decreasing''
  can be replaced by ``eventually decreasing'' in the above statement.
  This implies the assertion.
\end{proof}
The series
\[
    \sum_{n=2}^\infty \frac{(\log n!)^\alpha}{\big((\log n!)^\beta +r^2\big)^{\mu+1}}
\]
was considered in Corollary~7.1 of~\cite{GeHuTo19}. The corresponding sine series
can be analyzed analogously to the preceding propositions. By Stirling's formula,
\[
   \frac{(\log n!)^\alpha}{\big((\log n!)^\beta +r^2\big)^{\mu+1}}
   \sim
   (n \log n)^{-\theta}, \quad n\to\infty,
\]
and so the coefficients are regularly varying.
In particular, for
$\theta=\beta(\mu+1)-\alpha >2,$ we can proceed as in Proposition~\ref{prop:ha wi} to obtain
\[
    \sum_{n=2}^\infty \frac{(\log n!)^\alpha \sin(nx)}{\big((\log n!)^\beta +r^2\big)^{\mu+1}}
    \sim
    x\, \sum_{n=2}^\infty \frac{n(\log n!)^\alpha )}{\big((\log n!)^\beta +r^2\big)^{\mu+1}},
    \quad x\downarrow0.
\]
The paper~\cite{Te92} contains several estimates that can be applied to Mathieu sine series.
For instance, it was proved there (Corollary~2) that
\[
  \int_0^\pi \Big| \sum_{k=1}^n a_k \sin(kx) \Big| dx
  =\sum_{k=1}^n \frac{a_k}{k} + \mathrm{O}(a_1)
\]
for coefficient sequences $a_n\downarrow0$. This result can be readily applied to our
Mathieu sine series, with obvious constraints on the parameters to ensure monotonicity
of the coefficient sequence.
Moreover, Ul'yanov~\cite{Ul54} studied convergence in the $L^p$-quasinorm for $p\in(0,1)$,
for both sine and cosine series, which yields the following result.
\begin{proposition}
  Let $\mu,r\geq0$ and $\alpha,\beta,\gamma,\delta\in\mathbb{R}$ such that 
  $\theta=\beta(\mu+1)-\alpha\geq0.$
  If $\theta=0$, then we assume that  $\gamma-\delta(\mu+1)<0.$ We write
  $\tilde{S}_{\alpha,\beta,\mu}^{\gamma,\delta}(r,x;n)$ for the $n$-th partial sum of the
  series~\eqref{eq:math sin}. Then
  \[
    \lim_{n\to\infty} \int_{-\pi}^\pi \big|\tilde{S}_{\alpha,\beta,\mu}^{\gamma,\delta}(r,x)
    -\tilde{S}_{\alpha,\beta,\mu}^{\gamma,\delta}(r,x;n)\big|^p
    dx =0, \quad p\in(0,1).
  \]
  The same result holds for the corresponding cosine series.
\end{proposition}
\begin{proof}
  The coefficient sequence~\eqref{eq:an} eventually decreases. Therefore, it is
  of bounded variation, which by definition means that $\sum |\Delta a_n|<\infty.$
  Thus, both assertions are immediate from~\cite{Ul54}.
\end{proof}
Finally, we mention that $L^1$-convergence of the Mathieu-type cosine series
\begin{equation}\label{eq:math cos}
  S_{\alpha,\beta,\mu}^{\gamma,\delta}(r,x):=
  \sum_{n=2}^\infty \frac{n^\alpha (\log n)^\gamma\cos(nx)}
    {\big(n^\beta (\log n)^\delta+r^2\big)^{\mu+1}}.
\end{equation}
follows from a result in~\cite{Zi01}.
\begin{proposition}
  Let $\mu,r\geq0$ and $\alpha,\beta,\gamma,\delta\in\mathbb{R}$ such that either
  \[
    \alpha-\beta(\mu+1) < 0
  \]
  or
  \[
    \alpha=\beta=0 \quad \text{and} \quad \gamma-\delta(\mu+1)<-1.
  \]
  Then the series~\eqref{eq:math cos} converges in $L^1(0,\pi)$.
\end{proposition}
\begin{proof}
  As noted above, the coefficient sequence~\eqref{eq:an}
  of the series $S_{\alpha,\beta,\mu}^{\gamma,\delta}(r,x)$ is regularly varying. According to~\cite{Zi01},
  it then suffices to verify that $a_n \log n\to0$. But this easily follows from our assumption
  on the parameters.
\end{proof}
As for the asymptotic behavior of the cosine series ${S}_{\alpha,\beta,\mu}^{\gamma,\delta}(r,x)$
for $x\downarrow0$,
we can use Theorem~2.1 of~\cite{BoSe18} (going back to Zygmund)
in the case $0 < \theta=\beta(\mu+1)-\alpha <1.$ For~$\theta$ outside of this interval, the
other results for sine series we used in Propositions~\ref{prop:sin}
and~\ref{prop:ha wi} seem not to be available for cosine series so far.

\bibliographystyle{siam}
\bibliography{../gerhold}

\end{document}